\documentclass{article}
\usepackage{amsmath, amssymb}
\usepackage{tikz}
\def\qed{\ensuremath{\Box}}

\newcommand{\prop}{\ensuremath{\mathcal{P}}}

\newcounter{obs}
\newcounter{pro}
\newcounter{clm}

\newenvironment{proof}{\smallskip\par\noindent\textit{Proof. }}{\medskip\par}

\newenvironment{observation}{\refstepcounter{obs}\smallskip\par\noindent\textit{Observation \theobs.}}{\smallskip\par}

\newenvironment{proposition}{\refstepcounter{pro}\smallskip\par\noindent\textbf{Proposition \thepro.}}{\smallskip\par}

\newenvironment{claim}{\refstepcounter{clm}\smallskip\par\noindent(\theclm) }{\smallskip\par}

\newtheorem{theorem}{Theorem}
\newtheorem{lemma}{Lemma}
\newtheorem{corollary}{Corollary}

\begin{document}

\title{Strong chromatic index of chordless graphs}
\author{\begin{tabular}{p{.3\textwidth}}Manu Basavaraju\thanks{email: \texttt{manu.basavaraju@ii.uib.no}. The author was supported by European Research Council (ERC) Grant ``Rigorous Theory of Preprocessing'', reference 267959.}\\{\small\itshape University of Bergen, Norway}\end{tabular}
\and \begin{tabular}{p{.45\textwidth}}Mathew C. Francis\thanks{email: \texttt{mathew@isichennai.res.in}. The author would like to thank the support provided by the Institute of Mathematical Sciences, Chennai and the INSPIRE Faculty Award of the Department of Science and Technology, Government of India.}\\
{\small\itshape Indian Statistical Institute (ISI), Chennai}\end{tabular}}
\date{}

\maketitle

\begin{abstract}

A strong edge colouring of a graph is an assignment of colours to the edges of the graph such that for every colour, the set of edges that are given that colour form an induced matching in the graph.
The strong chromatic index of a graph $G$, denoted by $\chi'_s(G)$, is the minimum number of colours needed in any strong edge colouring of $G$.
A graph is said to be \emph{chordless} if there is no cycle in the graph that has a chord.
Faudree, Gy\'arf\'as, Schelp and Tuza~[The Strong Chromatic Index of Graphs, Ars Combin., 29B (1990), pp.~205--211] considered a particular subclass of chordless graphs, namely the class of graphs in which all the cycle lengths are multiples of four, and asked whether the strong chromatic index of these graphs can be bounded by a linear function of the maximum degree.
Chang and Narayanan~[Strong Chromatic Index of 2-degenerate Graphs, J. Graph Theory, 73(2) (2013), pp.~119--126] answered this question in the affirmative by proving that if $G$ is a chordless graph with maximum degree $\Delta$, then $\chi'_s(G) \leq 8\Delta -6$. We improve this result by showing that for every chordless graph $G$ with maximum degree $\Delta$, $\chi'_s(G)\leq 3\Delta$. This bound is tight up to to an additive constant.

\end{abstract}
\bibliographystyle{siam}

\section{Introduction}
A set of edges $M$ of a graph $G$ is said to be a \emph{matching} in $G$ if there are no two edges that share an end-point in $M$.
The set $M$ is said to be an \emph{induced matching} in $G$ if $M$ is a matching in $G$ and no two end-points that belong to different edges of $M$ are adjacent in $G$.

A strong edge colouring of a graph is an assignment of colours to the edges of the graph such that for every colour, the set of edges that are given that colour form an induced matching in the graph. It can also be defined to be a colouring of the edges of a graph such that every path in the graph uses three colours.
The strong chromatic index of a graph $G$, denoted by $\chi'_s(G)$, is the minimum number of colours needed in any strong edge colouring of $G$.
The concept of strong edge coloring was introduced by Fouquet and Jolivet~\cite{FouquetJolivet}.
Erd\"os and Ne\v set\v ril~\cite{ErdosNesetril} conjectured that for a graph $G$ with maximum degree $\Delta$,
$\chi'_s(G) = \frac{5}{4} \Delta^2$ if $\Delta$ is even and $\chi'_s(G) = \frac{1}{4} (5\Delta^2 -2\Delta+1)$ if $\Delta$ is odd.
They also gave examples where these bounds were tight. The best upper bound for $\chi'_s(G)$ is by Molloy and Reed~\cite{MolloyReed} who showed that $\chi'_s(G) \leq 1.998\Delta^2$.

The strong chromatic index of many special classes of graphs has also been studied. Faudree, Gy\'arf\'as, Schelp and Tuza~\cite{Faudree} proved that when $G$ is a graph in which all cycle lengths are multiples of four, $\chi'_s(G) \leq \Delta^2$. They asked whether this result can be improved to a linear function of the maximum degree. Brualdi and Quinn~\cite{Brualdi} improved the upper bound to $\chi'_s(G) \leq \alpha\beta$ for such graphs, where $\alpha$ and $\beta$ are the maximum degrees of the two partite sets in the graph (note that the graph is bipartite). Nakprasit~\cite{Nakprasit} proved that if $G$ is any bipartite graph and the maximum degree of one partite set is at most 2, then $\chi'_s(G) \leq 2\Delta$. 

A graph is said to be \emph{chordless} if there is no cycle in the graph that has a chord. Let $G$ be any graph such that all its cycle lengths are multiples of some fixed integer $t \ge 3$. Then it can easily be seen that the graph $G$ has to be chordless. Thus chordless graphs are a superclass of the class of graphs in which all the cycle lengths are multiples of four. Chang and Narayanan~\cite{Narayanan} proved that if $G$ is a chordless graph with maximum degree $\Delta$, then $\chi'_s(G) \leq 8\Delta -6$, thus answering the question of Faudree et al.~\cite{Faudree}. We improve this bound to $3\Delta$. In particular, we prove the following theorem.

\begin{theorem}\label{thm:main}
For any chordless graph $G$ with maximum degree $\Delta$, $\chi'_s(G)\leq 3\Delta$.
\end{theorem}

Clearly, the cycle on five vertices is an example of a graph $G$ for which $\chi'_s(G)= 3\Delta-1$. We show how, for every value of $\Delta$, one can construct graphs for which $\chi'_s$ is at least $3\Delta-2$.

\section{Preliminaries}\label{sec:prelim}

All the graphs considered in this paper are finite, simple and undirected.
Given a graph $G$, we denote its vertex set by $V(G)$ and its edge set by $E(G)$. We use the notation $|G|$ to denote $|V(G)|$ and $||G||$ to denote $|E(G)|$.
We say that $H$ is a \emph{non-trivial} subgraph of $G$ if $H$ is a subgraph of $G$ and $|H|\geq 2$.
Given $X\subseteq V(G)$, we denote by $G[X]$ the subgraph induced in $G$ by $X$. We use the notation $G-X$ to denote the graph $G[V(G)\setminus X]$. We often shorten $G-\{u\}$ to $G-u$. If $e\in E(G)$, then we let $G-e$ denote the graph with vertex set $V(G)$ and edge set $E(G)\setminus\{e\}$. The degree of a vertex $u$ in $G$ is denoted by $d_G(u)$.

A path $P$ in $G$ with $V(P)=\{v_0,v_1,\ldots,v_t\}$, where $v_i\neq v_j$ for $i\neq j$ and $E(P)=\{v_iv_{i+1}\colon 0\leq i\leq t-1\}$ is denoted by $v_0v_1\ldots v_t$. For $0\leq i\leq j\leq t$, the notation $v_iPv_j$ is shorthand for $v_iv_{i+1}\ldots v_j$.
For $X,Y\subseteq V(G)$, a path $P$ is said to be between $X$ and $Y$ if it is between a vertex in $X$ and a vertex in $Y$.
A cycle $C$ in $G$ with $V(C)=\{v_0,v_1,\ldots,v_t\}$ where $v_i\neq v_j$ for $i\neq j$ and $E(C)=\{v_iv_{i+1}\colon 0\leq i\leq t-1\}\cup\{v_tv_0\}$ is denoted as $v_0v_1\ldots v_tv_0$.

A graph $G$ is said to be $k$-connected, if there is no set of vertices $X\subseteq V(G)$ such that $G-X$ is disconnected and $|X|<k$.

Given a graph $G$, a subgraph $H$ of $G$ is said to be a \emph{block} of $G$ if $H$ is 2-connected and there is no subgraph $H'$ of $G$ such that $H'$ is 2-connected and $H$ is a subgraph of $H'$. Notice that a graph with two vertices and a single edge between them is considered to be 2-connected. Therefore, every edge in a graph forms a 2-connected subgraph of it.
We use the following version of Menger's Theorem~\cite{Diestel} for local 2-connectivity in graphs.

\begin{proposition}\label{pro:notincycle}
Let $G$ be a graph with $|G|>2$ and $a,b\in V(G)$ be two vertices such that $ab\notin E(G)$. Then there is no cycle in $G$ that contains both $a$ and $b$ if and only if there exists a vertex $x\in V(G)$ such that $a$ and $b$ are in different components of $G-x$.
\end{proposition}

For any notation that has not been defined here, please refer~\cite{Diestel}.
Now, we shall introduce some notation for the purposes of this paper.

Let $G$ be any graph and let $M$ denote a matching (not necessarily induced) in $G$. We denote by $G[M]$ the graph induced in $G$ by the endpoints of the edges in $M$. We denote by $G_M$ the graph obtained from $G[M]$ by contracting each edge $e\in M$ to a vertex labelled $v_e$ and removing all the parallel edges that result from this process.

\begin{observation}\label{obs:subgraph}
Let $G'_M$ be an induced subgraph of $G_M$ and let $M'=\{e\in M\colon v_e\in V(G'_M)\}$. Then, $G'_M=G_{M'}$.
\end{observation}

\begin{observation}\label{obs:expandpath}
Let $H$ be any subgraph of $G_M$.
If there is a path $v_{x_1y_1}v_{x_2y_2}\ldots v_{x_ky_k}$ in $H$, then there is a path $P$ in $G$ between $\{x_1,y_1\}$ and $\{x_k,y_k\}$ such that $V(P)\subseteq\{x_1,y_1,x_2,y_2,\ldots,x_k,y_k\}$ and $E(P)\subseteq \{x_iy_i\colon 1\leq i\leq k\}\cup \{x_ix_{i+1},x_iy_{i+1},y_ix_{i+1},y_iy_{i+1}\colon 1\leq i\leq k-1\}$.
\end{observation}

As we have assumed that $\Delta\geq 3$, we can use the following result of Machado, de Figueiredo and Trotignon~\cite{Trotignon}.

\begin{theorem}[\cite{Trotignon}]\label{thm:trotignon}
If $G$ is a chordless graph with maximum degree $\Delta\geq 3$, then $G$ can be edge-coloured with $\Delta$ colours.
\end{theorem}

\begin{lemma}\label{lem:contrgraph}
If for any matching $M$ in $G$, the graph $G_M$ can be vertex coloured with $k$ colours, then $G$ can be strong edge coloured with $k\chi'(G)$ colours, where $\chi'(G)$ is the chromatic index of $G$.
\end{lemma}
\begin{proof}
Let $f:E(G)\rightarrow\{1,2,\ldots,\chi'(G)\}$ be a proper edge colouring of $G$.
For every $i\in\{1,2,\ldots,\chi'(G)\}$, let $E_i$ denote the set of edges of $G$ that received the colour $i$. Clearly, for each $i$, $E_i$ is a matching in $G$. We know that the graph $G_{E_i}$ can be vertex coloured with $k$ colours.
For $e\in E(G)$, let $g(e)$ be the colour given to $v_e$ in the proper $k$-vertex colouring of $G_{E_{f(e)}}$.
We define a new colouring $h$ of the edges of $G$ as follows. For every edge $e\in E(G)$, we let $h(e)=(f(e),g(e))$. Consider any two edges $ab,cd\in E(G)$. Suppose $h(ab)=h(cd)=(i,j)$, then it implies that $v_{ab},v_{cd}\in V(G_{E_i})$ and that $v_{ab}$ and $v_{cd}$ got the same colour $j$ in the proper $k$-vertex colouring of $G_{E_i}$. As this means that $v_{ab}$ and $v_{cd}$ are nonadjacent in $G_{E_i}$, it follows that there is no edge in $G$ between $\{a,b\}$ and $\{c,d\}$. Therefore $h$ is a strong edge colouring of $G$. Since $h$ uses at most $k\chi'(G)$ colours, it follows that $\chi'_s(G)\leq k\chi'(G)$.
\hfill\qed
\end{proof}

The following corollary is a direct result of combining Theorem~\ref{thm:trotignon} and Lemma~\ref{lem:contrgraph}.
\begin{corollary}\label{cor:contrgraph}
If $G$ is a chordless graph with maximum degree $\Delta\geq 3$ and if for any matching $M$ in $G$, the graph $G_M$ can be vertex coloured with $k$ colours, then $\chi'_s(G)\leq k\Delta$.
\end{corollary}

\section{Proof of Theorem~\ref{thm:main}}

Let $G$ be any graph and let $M$ be a matching in it. We now define a colouring of the edges of $G_M$ using two colours.

\medskip

\noindent\textbf{Colouring the edges of $G_M$: }
We colour each edge of $G_M$ either red or blue as follows.
The edge $v_{ab}v_{cd}$ in $G_M$ is coloured with \emph{red} if there exist $\{\{p,q\},\{r,s\}\}=\{\{a,b\},\{c,d\}\}$  such that $pr\in E(G)$, $d_{G[M]}(p)=2$ and $d_{G[M]}(q)>2$.
Otherwise, the edge is coloured \emph{blue}.

From here onwards, whenever we consider a graph $G_M$, where $G$ is a graph and $M$ is a matching in $G$, we shall assume $G_M$ to have an implicit colouring of its edges as described above.

\begin{lemma}\label{lem:rededgestruc}
Let $G$ be a graph and $M$ be a matching in it. If $H$ is a subgraph of $G_M$ that contains only red edges,
then $||H||\leq |H|$.
\end{lemma}
\begin{proof}
Consider an edge $e=v_{ab}v_{cd}\in E(H)$. Since we coloured that edge red in $G_M$, we know that there exist $\{\{p,q\},\{r,s\}\}=\{\{a,b\},\{c,d\}\}$ such that $pr\in E(G)$, $d_{G[M]}(p)=2$ and $d_{G[M]}(q)>2$.
Define $f(e)=p$ and $g(e)=v_{pq}$.
Notice that $g$ is a function that maps every red edge in $H$ to one of its endpoints. This also means that for any edge $e\in E(H)$, $g(e)\in V(H)$. Now to prove that $||H||\leq |H|$ it suffices to prove that $g$ is injective. We shall prove this as follows.
Suppose there exist two edges $e_1,e_2\in E(H)$ such that $g(e_1)=g(e_2)=v_{pq}$. Let $e_1=v_{pq}v_{ab}$ and $e_2=v_{pq}v_{cd}$. Let us assume without loss of generality that $f(e_1)=p$. This implies that $d_{G[M]}(q)>2$. If $f(e_2)=q$, it would imply that $d_{G[M]}(q)=2$, which is a contradiction. Therefore, we have $f(e_2)=p$. But this would mean that in $G$, the neighbourhood of $p$ contains a vertex in $\{a,b\}$, a vertex in $\{c,d\}$ and $q$. This gives us $d_{G[M]}(p)\geq 3$, which is a contradiction to our assumption that $f(e_1)=p$.
Therefore, $g$ is injective.
\hfill\qed
\end{proof}

\begin{lemma}\label{lem:rededgedeg}
Let $G$ be any graph and $M$ be a matching in it. If $H$ is a nonempty subgraph of $G_M$ that contains only red edges, then either $H$ contains at least three vertices of degree at most 2 or $H$ is a single edge.
\end{lemma}
\begin{proof}
We shall first prove this lemma for the case when there are no degree 0 vertices in $H$.
Suppose there does not exist three vertices of degree at most 2 in $H$. 
Let $X$ denote the set of vertices of $H$ that have degree at most 2.
Since every vertex in $X$ has degree at least one in $H$, $||H||\geq (3(|H|-|X|)+|X|)/2$. Combined with Lemma \ref{lem:rededgestruc}, this gives $|H|\geq ||H||\geq (3(|H|-|X|)+|X|)/2$, which implies that $|H|\leq 2|X|$. As $|X|\leq 2$, we have $|H|\leq 4$. If $|H|=4$, then since $|X|\leq 2$, we have at least two vertices of degree 3 in $H$, which are in fact two vertices that are adjacent to every other vertex of $H$. But this would mean that $||H||\geq 5$, which is a contradiction. If $|H|=3$, then it cannot have a vertex of degree 3 and therefore $|X|=|H|=3$, which is again a contradiction. Therefore, it has to be the case that $|H|=2$, or, in other words, $H$ is a single edge.

Now let us consider the case when $H$ contains degree 0 vertices. Let $H'$ be the graph resulting from the removal of all degree 0 vertices from $H$. By the argument above, either there exist three vertices of degree at most 2 in $H'$ or $H'$ is a single edge. In the former case, the three vertices of degree at most 2 in $H'$ are also three vertices of degree at most 2 in $H$. In the latter case, the two vertices of $H'$ and any vertex of degree 0 in $H$ together gives three vertices of degree at most 2 in $H$.
\hfill\qed
\end{proof}

\medskip

\noindent\textbf{Property $\prop$}

\medskip

Let $G$ be a graph and $M$ a matching in it. $(G,M)$ is said to satisfy property $\prop$ if:
\begin{enumerate}
\renewcommand{\labelenumi}{(\roman{enumi})}
\itemsep 0in
\item $G_M$ is 2-connected, and
\item For every $ab\in E(G)\setminus M$ such that $a,b\in V(G[M])$, there is no cycle in $G-ab$ that contains both $a$ and $b$ (i.e., $ab$ is not a chord of a cycle in $G$).
\end{enumerate}

\begin{observation}\label{obs:subgraphprop}
Let $G$ be a graph and $M$ a matching in it such that $(G,M)$ satisfies the property $\prop$.
Let $M'\subseteq M$ such that $G_{M'}$ is 2-connected.
Then, $(G,M')$ also satisfies the property $\prop$.
\end{observation}

\begin{lemma}\label{lem:blueedgeconn}
Let $G$ be a graph and $M$ a matching in it such that $(G,M)$ satisfies the property $\prop$.
Let $e\in E(G_M)$ be an edge coloured blue.
Then, $G'_M=G_M-e$ is not 2-connected.
\end{lemma}
\begin{proof}
Let $e=v_{ab}v_{cd}$.
Suppose that the graph $G'_M$ is 2-connected.
Let $rs$ be any edge between $\{a,b\}$ and $\{c,d\}$ in $G$ ($r\in\{a,b\}$ and $s\in\{c,d\}$).
Let $G'=G-rs$.
Since $(G,M)$ satisfies the property $\prop$, we know that there is no cycle in $G'$ that contains both $r$ and $s$.\
From Proposition~\ref{pro:notincycle}, we know that there exists a vertex $x\in V(G')$ such that $r$ and $s$ are in different components of $G'-x$.
Let $G[R]$ be the component of $G'-x$ that contains $r$ and $G[S]$ the component of $G'-x$ that contains $s$.
Clearly, there is no edge in $G$ between two components of $G'-x$ other than $rs$.

Suppose there was an edge $pq\in M$ such that $p,q\in R$ and an edge $yz\in M$ such that $y,z\in S$.
Observe that every path in $G'$ between $R$ and $S$ passes through $x$.
Suppose there is a path in $G'_M$ between $v_{pq}$ and $v_{yz}$ that does not contain a vertex of the form $v_{xx'}$, then by Observation~\ref{obs:expandpath}, there is a path in $G$ between $\{p,q\}$ and $\{y,z\}$ that does not use the edge $rs$ (recall that $r\in\{a,b\}$, $s\in\{c,d\}$ and $v_{ab}v_{cd}\notin E(G'_M)$) or the vertex $x$. This would be a path in $G'$ between a vertex in $R$ and a vertex in $S$ that does not contain $x$, which is a contradiction. This means that every path in $G'_M$ between $v_{pq}$ and $v_{yz}$ should contain a vertex of the form $v_{xx'}$, with $xx'\in E(M)$. Therefore, the removal of $v_{xx'}$ breaks every path in $G'_M$ between $v_{pq}$ and $v_{yz}$.
This contradicts our assumption that $G'_M$ is 2-connected.

Therefore, we can assume without loss of generality that there is no edge in $G[R]$ that is also in $M$.
As $r\in V(G[M])$ and $rs\not\in M$ is the only edge between $G[R]$ and any other component of $G'-x$, this leaves us with only one possibility: $rx=ab\in M$. It can be seen that $R\cap V(G[M])=\{r\}$ in the following way. Suppose there exists $u\in R\cap V(G[M])$ with $u\neq r$. Clearly, there also exists $w\in V(G)$ such that $uw\in M$. As $M$ is a matching and $rx\in M$, we have $w\neq r$, $w\neq x$ and because there is no edge in $R$ that is also in $M$, we have $w\notin R$. But $w$ cannot be in any component other than $G[R]$ in $G'-x$ as we know that $rs$ is the only edge between two components of $G'-x$. This shows that $R\cap V(G[M])=\{r\}$.

This means that the vertex $r$ had degree 2 in $G[M]$, the two edges incident on it being $rx$ and $rs$. Note that if $x$ had degree less than 3 in $G[M]$, then the degree of $v_{rx}=v_{ab}$ would be at most 2 in $G_M$ and therefore at most 1 in $G'_M$, contradicting the assumption that $G'_M$ is 2-connected. This means that $x$ has degree at least 3 in $G[M]$. But then, the edge $v_{ab}v_{cd}$ in $G_M$ would not have been coloured blue. Therefore, we can conclude that $G'_M$ is not 2-connected.

\hfill\qed
\end{proof}

Let $H$ be a graph that is not 2-connected. A block $B$ of $H$ is said to be a \emph{leafblock} of $H$ if $B$ contains exactly one cutvertex of $H$. Note that leafblocks correspond to the leaves of the block graph of $H$ \cite{Diestel}. The following propositions are easy to prove. The interested reader may refer to the appendix for their proofs.

\begin{proposition}\label{pro:leafblocks}
If $G$ is 2-connected and $e$ is an edge in $G$ such that $G-e$ is not 2-connected, then:
\begin{enumerate}
\renewcommand{\labelenumi}{(\roman{enumi})}
\renewcommand{\theenumi}{(\roman{enumi})}
\itemsep 0in
\item\label{proit:endpnotcut} the endpoints of $e$ are not cutvertices of $G-e$,
\item\label{proit:endpnotincycle} the endpoints of $e$ are not contained in any cycle of $G-e$, and
\item\label{proit:twoleafblocks} except for the case when $G$ is a single edge, $G-e$ has exactly two leafblocks that contain one end-point each of $e$.
\end{enumerate}
\end{proposition}

\begin{proposition}\label{pro:remleafblock}
Let $G$ be a graph that is connected, but not 2-connected. Let $B_1,B_2,\ldots,B_k$ be some leafblocks of $G$ and let $u_1,u_2,\ldots,u_k$ denote the respective cutvertices of $G$ that are in these leafblocks. Then, if $X=\bigcup_{i=1}^k (B_i\setminus\{u_i\})$, then $G-X$ is connected.
\end{proposition}

\begin{corollary}\label{cor:blueedgeconn}
Let $G$ be any graph and $M$ a matching in it such that $(G,M)$ satisfies the property $\prop$.
Let $e\in E(G_M)$ be an edge coloured blue in $G_M$ and let $H$ be a 2-connected subgraph of $G_M$ such that $e\in E(H)$.
Then, $H-e$ is not 2-connected.
\end{corollary}
\begin{proof}
We know by Lemma~\ref{lem:blueedgeconn} that $G_M-e$ is not 2-connected. From Proposition~\ref{pro:leafblocks}\ref{proit:endpnotincycle}, we know that there is no cycle in $G_M-e$ that contains both the endpoints of $e$. This implies that there is no 2-connected subgraph in $G_M-e$ that contains both the endpoints of $e$. Therefore $H-e$ cannot be 2-connected.
\hfill\qed
\end{proof}

\begin{proposition}\label{pro:2connpaths}
If $G$ is 2-connected and $x,y,z\in V(G)$ with $y\neq z$, then in $G$, there exist a path $P_1$ from $x$ to $y$ and a path $P_2$ from $x$ to $z$ such that $V(P_1)\cap V(P_2)=\{x\}$.
\end{proposition}

\begin{lemma}\label{lem:twodeg2prop}
Let $G$ be any graph and $M$ be a matching in it such that $(G,M)$ satisfies the property $\prop$. Let $H$ be any 2-connected subgraph of $G_M$ that is not a single edge. Then:
\begin{enumerate}
\item If $e\in E(H)$ is any edge coloured blue in $G_M$, then $H-e$ is not 2-connected, it has exactly two leafblocks and each of the two leafblocks contains a vertex with degree at most 2 in $H$ that is not a cutvertex of $H-e$.
\item If there are no blue edges in $H$, then there exist three vertices of degree at most 2 in $H$.
\end{enumerate}
\end{lemma}
\begin{proof}
We shall prove this by induction on the number of blue edges in $H$. If there are no blue edges, then by Lemma \ref{lem:rededgedeg}, we know that there are three vertices of degree at most 2 in $H$. We shall therefore assume that $H$ contains at least one blue edge, say $e$.

Let $e=ab$.
As $(G,M)$ satisfies the property $\prop$, we know by Corollary \ref{cor:blueedgeconn} that $H-e$ is not 2-connected. By Proposition~\ref{pro:leafblocks}\ref{proit:twoleafblocks}, $H-e$ has exactly two leafblocks, which we shall call $A$ and $B$. Let $a'$ denote the cutvertex of $H-e$ that is in $A$ and $b'$ denote the cutvertex of $H-e$ that is in $B$. By Propositions \ref{pro:leafblocks}\ref{proit:twoleafblocks} and \ref{pro:leafblocks}\ref{proit:endpnotcut}, we can assume without loss of generality that $a\in V(A)\setminus\{a'\}$ and $b\in V(B)\setminus\{b'\}$.
We shall now show that there exists a vertex in $A$ with degree 2 in $H$ that is not $a'$. The same arguments can be used to show that there exists a vertex in $B$ with degree 2 in $H$ that is not $b'$.

Let us consider the block $A$.

If $A$ is a single edge, then since $a\neq a'$, $a$ is a degree 1 vertex in $H-e$ and therefore a degree 2 vertex in $H$. So, we have a vertex in $A$ with degree 2 in $H$ that is not $a'$.

Suppose $A$ is not a single edge.

If every edge in $A$ is coloured red, then by Lemma \ref{lem:rededgedeg}, we know that there exists $X\subseteq V(A)$ with $|X|=3$ such that every vertex in $X$ has degree at most 2 in $A$. We therefore have $X\setminus\{a,a'\}\neq\emptyset$. Let $z\in X\setminus\{a,a'\}$. Then $z$ has degree 2 in $H-e$ and also in $H$. Thus, we again have a vertex of the required kind.

Suppose $A$ contains a blue edge, say $e'=xy$.
Since $A$ is a 2-connected subgraph of $G_M$ that is not a single edge and which
contains fewer number of blue edges than $H$, we can assume by the induction hypothesis that $A-e'$ is not 2-connected, it has exactly two leafblocks and that each of these two leafblocks contain a vertex with degree 2 in $A$ that is not a cutvertex of $A-e'$. By Proposition~\ref{pro:leafblocks}\ref{proit:twoleafblocks}, we know that $x$ and $y$ are in different leafblocks of $A-e'$. Let the leafblock of $A-e'$ containing $x$ be $X$ and the leafblock of $A-e'$ containing $y$ be $Y$. Also, let $x'$ denote the cutvertex of $A-e'$ in $X$ and let $y'$ denote the cutvertex of $A-e'$ in $Y$. Note that by Proposition~\ref{pro:leafblocks}\ref{proit:endpnotcut}, we have $x\neq x'$ and $y\neq y'$. Let the two vertices in $A-e'$ with degree 2 in $A$ that are guaranteed to exist by the induction hypothesis be $u\in V(X)\setminus\{x'\}$ and $w\in V(Y)\setminus\{y'\}$.
The proof will be completed if we show that $\{u,w\}\neq\{a,a'\}$ (this is because every vertex in $A$ other than $a$ and $a'$ has the same degree in both $A$ and $H$). For this purpose, we shall show that it is not possible that one of $a,a'$ is in $V(X)\setminus\{x'\}$ and the other in $V(Y)\setminus\{y'\}$. Note that once we show this, it follows that $\{u,w\}\neq\{a,a'\}$ as we already know that $u\in V(X)\setminus\{x'\}$ and $w\in V(Y)\setminus\{y'\}$.

Suppose for the sake of contradiction that one of $a,a'$ is in $V(X)\setminus\{x'\}$ and the other in $V(Y)\setminus\{y'\}$. We shall assume without loss of generality that $a\in V(X)\setminus\{x'\}$ and $a'\in V(Y)\setminus\{y'\}$. By Proposition \ref{pro:2connpaths}, we know that there exist paths $P_{xa}$ from $x$ to $a$ and $P_{xx'}$ from $x$ to $x'$ in $X$ such that $P_{xa}$ and $P_{xx'}$ do not meet in any vertex other than $x$. Similarly, there exist paths $P_{ya'}$ from $y$ to $a'$ and $P_{yy'}$ from $y$ to $y'$ in $Y$ such that $P_{ya'}$ and $P_{yy'}$ do not meet at any vertex other than $y$.
By Proposition~\ref{pro:remleafblock}, we know that $(A-e')-((V(X)\setminus\{x'\})\cup (V(Y)\setminus\{y'\}))$ is also connected (note that $A-e'$ is connected, as $A$ is not a single edge). Therefore, there is a path $P_{x'y'}$ in $(A-e')-((V(X)\setminus\{x'\})\cup (V(Y)\setminus\{y'\}))$ from $x'$ to $y'$. Because of the same reason, there is a path $P_{a'b}$ in $(G_M-e)-(V(A)\setminus\{a'\})$ from $a'$ to $b$ (again, $G_M$ cannot be a single edge as it has a 2-connected subgraph $H$ that is not a single edge and therefore, $G_M-e$ is connected). Now, since $ab\in E(G_M)$, we have a cycle $xP_{xx'}x'P_{x'y'}y'P_{y'y}yP_{ya'}a'P_{a'b}baP_{ax}x$ in $G_M-e'$, where $P_{y'y}$ and $P_{ax}$ denote the reversals of the paths $P_{yy'}$ and $P_{xa}$ respectively. But this is a cycle containing $x$ and $y$ in $G_M-e'$. As $e'$ was a blue edge in $G_M$, by Lemma \ref{lem:blueedgeconn}, we know that $G_M-e'$ is not 2-connected and by Proposition~\ref{pro:leafblocks}\ref{proit:endpnotincycle}, we know that there cannot 
be a cycle in $G_M-e'$ that contains both $x$ and $y$. Thus, we have a contradiction.
\hfill\qed
\end{proof}

\begin{proposition}\label{pro:deg2block}
Let $G$ be any graph such that every 2-connected subgraph of $G$ contains at least two vertices of degree at most 2. Then, every non-trivial subgraph of $G$ contains at least two vertices with degree at most 2.
\end{proposition}

\begin{theorem}\label{thm:twin2deg}
Let $G$ be any graph and $M$ be a matching in it such that $(G,M)$ satisfies the property $\prop$. Every non-trivial subgraph of $G_M$ contains at least two vertices with degree at most 2.
\end{theorem}
\begin{proof}
Let $G'_M$ be an induced subgraph of $G_M$. By Observation~\ref{obs:subgraph}, there exists a matching $M'\subseteq M$ in $G$ such that $G'_M=G_{M'}$.

If $G_{M'}$ is 2-connected, then by Observation~\ref{obs:subgraphprop}, $(G,M')$ also satisfies the property $\prop$. If $G_{M'}$ is a single edge, then clearly, it contains two vertices with degree at most 2. If $G_{M'}$ is not a single edge, it follows from Lemma~\ref{lem:twodeg2prop} that $G_{M'}$ contains at least two vertices with degree at most 2. This shows that every 2-connected induced subgraph, and therefore every 2-connected subgraph, of $G_M$ contains at least two vertices with degree at most 2. Now, from Proposition~\ref{pro:deg2block}, we know that every non-trivial subgraph of $G_M$ contains at least two vertices with degree at most 2.
\hfill\qed
 
\end{proof}

\begin{corollary}
Let $G$ be any graph and $M$ a matching in it such that $(G,M)$ satisfies the property $\prop$. Then $G_M$ is 2-degenerate, and therefore 3-colourable.
\end{corollary}

\begin{theorem}
If $G$ is a chordless graph and $M$ a matching in it, then every non-trivial subgraph of $G_M$ contains at least two vertices of degree at most 2.
\end{theorem}
\begin{proof}
Consider any 2-connected induced subgraph $G'_M$ of $G_M$. From Observation \ref{obs:subgraph}, there exists $M'\subseteq M$ such that $G'_M=G_{M'}$.
If $ab\in E(G)$ and $G-ab$ contains a cycle that contains both $a$ and $b$, then clearly, $ab$ was a chord of that cycle in $G$. But this is a contradiction as we know that $G$ is chordless. Therefore, $(G,M')$ satisfies the property $\prop$. Now, from Theorem \ref{thm:twin2deg}, we know that $G_{M'}$ contains at least two vertices with degree at most 2. Thus, any 2-connected induced subgraph of $G_M$ contains at least two vertices with degree at most 2 implying that every 2-connected subgraph contains at least two vertices with degree at most 2.
Now, from Proposition \ref{pro:deg2block}, it is clear that every non-trivial subgraph of $G_M$ contains at least two vertices of degree at most 2.\hfill\qed
\end{proof}
\begin{corollary}\label{cor:main}
If $G$ is a chordless graph and $M$ a matching in it, then $G_M$ is 2-degenerate and therefore 3-colourable.
\end{corollary}


The statement of Theorem~\ref{thm:main} can be easily verified to be true if $\Delta\leq 2$ as in this case, $G$ is a disjoint union of cycles and paths, which can be strong edge-coloured with at most 5 colours. When $\Delta\geq 3$, Corollaries~\ref{cor:main} and~\ref{cor:contrgraph} together give $\chi'_s(G)\leq 3\Delta$. This proves Theorem~\ref{thm:main}.

\bigskip

\noindent\textbf{Tightness: }
A cycle on five vertices is an example of a graph for which the strong chromatic index is $3\Delta-1$. However, for $\Delta\geq 3$, we do not know of a graph for which strong chromatic index is at least $3\Delta-1$. But for any value of $\Delta\geq 2$, we can show a family of graphs for which the strong chromatic index is at least $3\Delta-2$ as follows. Consider the graph obtained by taking a $K_{2,\Delta}$ and adding $\Delta-2$ pendant vertices to one vertex of degree two so that its degree now becomes exactly $\Delta$. This graph is shown in Figure~\ref{fig:tightness}. In the figure, the vertices $a$, $b$ and $c$ have degree $\Delta$. It can be easily verified that every edge will have to be given a different colour in any strong edge colouring of this graph. Since the number of edges in this graph is $3\Delta-2$, the strong chromatic index of this graph is $3\Delta-2$. It can also be seen that any chordless graph $G$ with maximum degree $\Delta$ that contains this graph as a subgraph will have $\chi'_s(G)
\geq 3\Delta -2$.

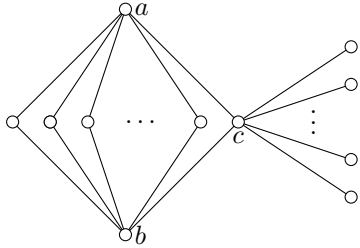
\begin{figure}
\centering
\begin{tikzpicture}
\coordinate (top) at (2,2);
\coordinate (mid1) at (.5,.5);
\coordinate (midlast) at (3.5,.5);
\coordinate (bottom) at (2,-1);
\coordinate (mid2) at (1,.5);
\coordinate (mid3) at (1.5,.5);
\coordinate (midlastbutone) at (3,.5);
\node (dots) at (2.2,.5) {$\ldots$};
\coordinate (fan1) at (5,1.5);
\coordinate (fan2) at (5,1);
\coordinate (fanlastbutone) at (5,0);
\coordinate (fanlast) at (5,-.5);
\node (vdots) at (4.5,.6) {$\vdots$};

\node [right] (labela) at (top) {$a$};
\node [right] (labelb) at (bottom) {$b$};
\node [below] (labelc) at (midlast) {$c$};

\draw (top) -- (mid1) -- (bottom);
\draw (top) -- (mid2) -- (bottom);
\draw (top) -- (mid3) -- (bottom);
\draw (top) -- (midlastbutone) -- (bottom);
\draw (top) -- (midlast) -- (bottom);
\draw (midlast) -- (fan1);
\draw (midlast) -- (fan2);
\draw (midlast) -- (fanlastbutone);
\draw (midlast) -- (fanlast);

\filldraw [fill=white] (top) circle(.08cm);
\filldraw [fill=white] (mid1) circle(.08cm);
\filldraw [fill=white] (mid2) circle(.08cm);
\filldraw [fill=white] (mid3) circle(.08cm);
\filldraw [fill=white] (midlast) circle(.08cm);
\filldraw [fill=white] (midlastbutone) circle(.08cm);
\filldraw [fill=white] (bottom) circle(.08cm);
\filldraw [fill=white] (mid2) circle(.08cm);
\filldraw [fill=white] (fan1) circle(.08cm);
\filldraw [fill=white] (fan2) circle(.08cm);
\filldraw [fill=white] (fanlastbutone) circle(.08cm);
\filldraw [fill=white] (fanlast) circle(.08cm);
\end{tikzpicture}
\caption{A chordless graph with strong chromatic index at least $3\Delta-2$. The vertices $a$, $b$ and $c$ have degree equal to $\Delta$.}
\label{fig:tightness}
\end{figure}
\section{Conclusion}
A graph is \emph{minimally 2-connected} if it is 2-vertex-connected and the removal of any edge from the graph makes its vertex connectivity less than two.
They are also exactly the 2-connected graphs that are also chordless~\cite{Dirac}.  Theorem~\ref{thm:main} therefore implies the following corollary.

\begin{corollary}
If $G$ is any minimally 2-connected graph with maximum degree $\Delta$, $\chi'_s(G)\leq 3\Delta$.
\end{corollary}

Notice that the graph in Figure~\ref{fig:tightness} is not 2-connected. It can be asked whether for minimally 2-connected graphs, the upper bound given by Theorem~\ref{thm:main} can be improved. We pose this as an open question.

\medskip

\noindent\textbf{Question. }
Can an upper bound of the form $2\Delta+c$, where $c$ is some constant, be found for the strong chromatic index of minimally 2-connected graphs?

\newpage
\appendix
\section{Proofs of Propositions}
Let $G$ be a graph that is not 2-connected. The following statements are folklore.

\medskip
\begin{claim}\label{clm:twoleafblocks}
If $G$ is connected (but not 2-connected) and $|V(G)|>1$, then $G$ has at least two leafblocks.
\end{claim}
\begin{claim}\label{clm:leafblockcomponent}
If $B$ is a leafblock of $G$ and $u$ is the cutvertex of $G$ in $B$, then $B-u$ is a component of $G-u$.
\end{claim}
\begin{claim}\label{clm:cutvertexinpath}
Let $B$ be a leafblock of $G$ and $u$ the cutvertex of $G$ in $B$. If $P$ is a path between a vertex $v_1\in V(G)\setminus (B\setminus\{u\})$ and $v_2\in B\setminus \{u\}$, then $P$ contains $u$.
\end{claim}

\medskip

Let $G$ be a 2-connected graph and let $e$ be an edge of $G$ such that $G-e$ is not 2-connected.

\smallskip

\begin{claim}\label{clm:edgeacrosscutvertex}
If $u$ is any cutvertex of $G-e$, then $(G-e)-u$ has exactly two components and there is one endpoint of $e$ in each of them.
\end{claim}
\begin{proof}
Since $u$ is not a cutvertex of $G$, we know that $G-u$ is a connected graph. Therefore, adding the edge $e$ back to the graph $(G-e)-u$ should make the graph connected. It can be easily seen that this is possible only if $(G-e)-u$ contained exactly two components, each containing one endpoint of $e$.
\hfill\qed
\end{proof}

\begin{claim}\label{clm:endpointleafblock}
If $B$ is a leafblock of $G-e$ with $u$ being the unique cutvertex of $G-e$ in $B$, then there is exactly one endpoint of $e$ in $B-u$.
\end{claim}
\begin{proof}
Let $u$ be the unique cutvertex of $G-e$ that is contained in $B$. From (\ref{clm:leafblockcomponent}), we know that $B-u$ is a component of $(G-e)-u$.
By (\ref{clm:edgeacrosscutvertex}), we know that each component of $(G-e)-u$ contains exactly one endpoint of $e$. 
Therefore, there is exactly one endpoint of $e$ in $B-u$.
\hfill\qed
\end{proof}

\medskip

We shall now prove Propositions \ref{pro:leafblocks}, \ref{pro:2connpaths} and \ref{pro:deg2block}.

\medskip

\noindent\textbf{Proof of Proposition \ref{pro:leafblocks}: }
If $G$ is a single edge, then $e$ has to be that single edge, and statements \ref{proit:endpnotcut} and \ref{proit:endpnotincycle} of the proposition are easily seen to be true. In this case, the statement \ref{proit:twoleafblocks} is not relevant.

Suppose $G$ is not a single edge.
It can be easily seen that $G-e$ is a connected graph as otherwise, $G$ would not have been 2-connected. Also $G-e$ has at least two vertices as it should contain the two endpoints of $e$. Therefore, from (\ref{clm:twoleafblocks}), we know that there are at least two leafblocks in $G-e$. Combining this with (\ref{clm:endpointleafblock}), it is clear that there are exactly two leafblocks in $G-e$, with the two endpoints of $e$ being in different leafblocks. Note that (\ref{clm:endpointleafblock}) also tells us that the endpoints of $e$ are not the cutvertices in the respective leafblocks in which they occur. We have thus proved \ref{proit:twoleafblocks} and \ref{proit:endpnotcut}. 

Let $B_1$ and $B_2$ be the two leafblocks of $G-e$ and let $u_1$ and $u_2$ be the cutvertices of $G-e$ in $B_1$ and $B_2$ respectively. Let $e=xy$ and assume without loss of generality that $x\in V(B_1)\setminus \{u_1\}$ and $y\in V(B_2)\setminus\{u_2\}$. Clearly, from (\ref{clm:leafblockcomponent}), $x$ and $y$ are in different components of $(G-e)-u_1$, and therefore, from Proposition~\ref{pro:notincycle}, there can be no cycle that contains $x$ and $y$ in $G-e$.
Thus, we have proved \ref{proit:endpnotincycle}.\hfill\qed

\medskip

\noindent\textbf{Proof of Proposition \ref{pro:remleafblock}: }
We claim that if $P$ is any path in $G$ between vertices $v_1,v_2\in V(G-X)$, then $P\cap \bigcup_{i=1}^k (B_i\setminus\{u_i\})=\emptyset$. Suppose not. Let $x\in (B_t\setminus\{u_t\})$ be a vertex in $P$. Then, $P$ is of the form $v_1PxPv_2$. From (\ref{clm:cutvertexinpath}), $v_1Px$ has to be of the form $v_1Pu_tPx$ and $xPv_2$ has to be of the form $xPu_tPv_2$. This means that the vertex $u_t$ occurs twice in $P$, which contradicts the assumption that $P$ is a path. This, combined with the fact that $G$ is connected, shows that between any two vertices in $V(G-X)$, we have a path in $G$ that does not use any vertex in $\bigcup_{i=1}^k (B_i\setminus\{u_i\})$. This path remains a path in the graph $G-X$ and therefore, the graph $G-X$ is connected.
\hfill\qed

\medskip

\noindent\textbf{Proof of Proposition \ref{pro:2connpaths}: }
If either $y=x$ or $z=x$, there is nothing to prove as one of the paths can be taken to be just $\{x\}$.
Suppose this is not the case. Since $G$ is 2-connected, we have $|N(x)|\geq 2$.
Then, by Menger's Theorem, there exist at least two vertex disjoint paths between $N(x)$ and $\{y,z\}$ in $G$ that do not contain $x$. The vertex $x$ can be appended to each of these paths to obtain the required paths $P_1$ and $P_2$.
\hfill\qed

\medskip

\noindent\textbf{Proof of Proposition \ref{pro:deg2block}: }
Let $H$ be any non-trivial subgraph of $G$.

If $H$ is not connected, then consider the largest connected component $H'$ of $H$. If $|V(H')|=1$, then every component of $H$ consists of a single vertex. Since $H$ has at least two connected components, and because the vertices that make up each of these components are vertices with degree 0 in $H$, any two of these vertices are two vertices with degree at most 2 in $H$. If $|V(H')|>1$, then continue the argument setting $H=H'$ (as $H'$ is connected) to find two vertices with degree at most 2 in $H'$, and these will also have degree at most 2 in $H$.

If $H$ is 2-connected, then the statement of the proposition already tells us that $H$ contains at least two vertices of degree at most 2.
If $H$ is connected and not 2-connected, then we know by (\ref{clm:twoleafblocks}) that $H$ contains at least two leafblocks, say $B_1$ and $B_2$ (recall that $H$ is non-trivial). Let us consider $B_1$. Since $B_1$ is a 2-connected subgraph of $G$, we know that it contains at least two vertices of degree at most 2. As at most one of these vertices could be the cutvertex of $H$ in $B_1$, we can always choose a vertex $u$ from these two vertices such that $u$ is not a cutvertex of $H$. Clearly, every neighbour of $u$ in $H$ is in $B_1$, which implies that $d_H(u)=d_{B_1}(u)=2$. We can similarly find a vertex of degree at most 2 in $H$ in $B_2$, and it is clear that this vertex will be distinct from $u$. We have thus shown that there exist two vertices of degree at most 2 in $H$.
\hfill\qed

\end{document}